\documentclass[a4paper,12pt,twoside]{article}
\usepackage[utf8]{inputenc}
\RequirePackage[T1]{fontenc}
\usepackage{amsmath,mathrsfs,amssymb,amsthm,textcomp}
\usepackage[colorlinks,bookmarksopen,bookmarksnumbered,citecolor=blue,urlcolor=red]{hyperref}
\usepackage[english]{babel}

\newcounter{ass}

\RequirePackage{amsmath}
\RequirePackage{amssymb}
\def\O{\Omega}
\def\s{\sigma}

\def\d{\delta} 
\def\hat{\widehat}

\def\disp{\displaystyle}

\def\e{\eta}
\def\E{\mathbb{E}}
\def\Prb{\mathbb{P}}\def\P{\mathbb{P}}
\def\tl{\widetilde}

\def\R{\mathbb{R}}

\newtheorem{theorem}{Theorem}[section]
\newtheorem{definition}[theorem]{Definition}
\newtheorem{proposition}[theorem]{Proposition}
\newtheorem{lemma}[theorem]{Lemma}
\newtheorem{corollary}[theorem]{Corollary}
\newtheorem{remark}[theorem]{Remark}
\numberwithin{equation}{section}
\begin{document}
\title{\textbf{Averaging Principle for Backward Stochastic Differential Equations driven by both standard and fractional Brownian motions }}
 \author{\centering  $\textrm{Ibrahima Faye}^{1}$,  $\textrm{Sadibou Aidara}^{2}$, $\textrm{Yaya Sagna}^{3}$ }
 
\date{ }
\maketitle
\begin{center}
$	^{1}$ Universit\'e Alioune Diop de Bambey, BP. 30 Bambey\\
ibou.faye@uadb.edu.sn\\
$	^{2} \& ^{3}$ LERSTAD, UFR Sciences Appliqu\'ees et de Technologie,\\
Universit\'e Gaston Berger, BP 234, Saint-Louis, SENEGAL\\
sadibou.aidara.ugb@gmail.com \& sagnayaya88@gmail.com\\
\end{center}
\begin{abstract}
\footnotesize{\noindent Stochastic averaging for a class of backward stochastic differential equations driven by both standard and fractional Brownian motions (SFrBSDEs in short), is investigated. An averaged SFrBSDEs for the original SFrBSDEs is proposed, and their solutions are quantitatively compared. Under some appropriate assumptions, the solutions to original systems can be approximated by the solutions to averaged stochastic systems in the sense of mean square and also in probability.}
	\end{abstract}	
\textbf{Mathematic Subjection Classification:} 60H05, 60G44
\medspace\\
\textbf{keywords}: Averaging principle, backward stochastic differential equation, Stochastic calculus, fractional Brownian motion, Chebyshev's inequality  and It\^o's representation formula.

\section{Introduction}

Backward stochastic differential equations (BSDEs in short) were first introduced by Pardoux and Peng \cite{Par-Peng90} with Lipschitz assumption under which they proved the celebrated existence and uniqueness result. This pioneer work was extensively used in many fields like stochastic interpretation of solutions of PDEs and financial mathematics. Few years later, several authors investigated  BSDEs with respect to fractional Brownian motion $\left(B^H_t\right)_{t\geq 0}$ with Hurst parameter $H$. This process is a self-similar, i.e. $B^H_{at}$ has the same law as $a^HB^H_t$ for any $a>0$, it has a long range dependence for $H>\frac{1}{2}$. For $H=\frac{1}{2}$ we obtain a standard Wiener process, but for $H\neq\frac{1}{2}$, this process is not a semimartingale. These properties make this process a useful driving noise in models arising in physics, telecommunication networks, finance and other fields.

Bender \cite{Ben} gaves  one  of the  earliest  result  on  fractional BSDEs (FrBSDEs in short). The author established an explicit solution of a class of linear FrBSDEs with arbitrary  Hurst parameter $H$. This is done essentially by means  of solution of a specific  linear parabolic PDE.  
There are two major obstacles depending on the properties of fractional Brownian motion: Firstly, the fractional Brownian motion is not a semimartingale except for the case of  Brownian motion ($H=\frac{1}{2}$), hence the classical It\^o calculus based on semimartingales cannot be transposed directly to the fractional case. Secondly, there is no martingale representation theorem with respect to the fractional Brownian motion. Studing nonlinear fractional BSDEs, Hu and Peng \cite{Hu-Peng} overcame successfully the second obstacle in  the case $H>\frac{1}{2}$ by means of the quasi-conditional expectation. 
The authors prove existence and uniqueness of the solution but with some restrictive assumptions on  the  generator. In  this same spirit, Maticiuc and Nie \cite{Mat-Nie}  interesting in  backward stochastic variational inequalities, improved this first  result by weakening the required condition  on  the drift of  the stochastic equation. Fei et al \cite{Fei} introduced the following type of  BSDEs  driven by both standard and fractional Brownian motions (SFrBSDEs in short)
\begin{equation} \label{SFrBSDEs}  
Y_t = \xi + \!\int_t^Tf(s, \e_s, Y_s, Z_{1,s},Z_{2,s})ds -\!\int_t^TZ_{1,s}dB_s - \!\int_t^TZ_{2,s}dB_s^H, \quad 0\leq t\leq T,
\end{equation}
where $\left(B_t\right)_{t\geq 0}$ is a standard Brownian  motion, $\left(B^H_t\right)_{t\geq 0}$ is a  fractional  Brownian  motion and $\left\{\e_t\right\}_{0\le t\le T}$ is a solution of a stochastic differential equation driven by both standard and fractional Brownian motions. In \cite{Fei}, the authors abtained the existence and uniqueness of the solution of SFrBSDEs under Lipschitz assumptions. Recently, new classes of BSDEs driven by two mutually independent fractional Brownian motions were introduced by Aidara and Sagna \cite{SA19}. They established the existence and uniqueness of solutions.

Stochastic averaging principle, which is usually used to approximate dynamical systems under random fluctuations, has long and rich history in multiscale problems (see, e.g.,\cite{A78}). Recently, the averaging principle for BSDEs and one-barrier reflected BSDEs, with Lipschitz coefficients,  were first studied  by Jing and Li \cite{JL21}.  In the present paper, we study a stochastic averaging technique for a class of the SFrBSDEs \eqref{SFrBSDEs}.  We present an averaging principle, and prove that the original  SFrBSDEs can be approximated by an averaged SFrBSDEs in the sense of mean square convergence and convergence in probability, when a scaling parameter tends to zero.

The rest of the paper is arranged as follows. In Section 2, we recall some definitions and results about fractional stochastic integrals and the related It\^o formula. In Section 3, we investigate the averaging principle for the SFrBSDEs under some proper conditions.

\section{Fractional Stochastic calculus}
Let $\O$ be a non-empty set,  $\mathcal{F}$	a $\sigma-$algebra of sets $\O$, $\P$ a probability measure defined on $\mathcal{F}$ and  $\left\lbrace \mathcal{F}_t,\; t\in[0,T]\right\rbrace$ a $\sigma-$algebra generated by both standard and fractional Brownian motions.
The triplet $(\O, \mathcal{F}, \P)$ defines  a probability space and $\E$  the mathematical expectation with respect to the  probability measure $\P$.

The fractional  Brownian  motion $\left(B^H_t\right)_{t\geq 0}$  with Hurst parameter $H\in(0,1)$ is a zero mean Gaussian process with the covariance function
$$\E[B_t^HB_s^H]=\frac{1}{2}\left(t^{2H} + s^{2H} -|t-s|^{2H}\right), \quad t,s\geq0.$$
Suppose that the process $\left(B^H_t\right)_{t\geq 0}$ is independent of the standard Brownian motion $\left(B_t\right)_{t\geq 0}$.
Throughout this paper it is assumed that $H\in (1/2, 1)$ is arbitrary but fixed.

Denote $\rho(t,s) =  H(2H-1)|t-s|^{2H-2}, \; (t,s)\in\R^2.$
Let $\xi$ and $\eta$ be measurable functions on  $[0,T]$. Define
$$\langle \xi,\eta\rangle_t =\int_0^t\int_0^t\rho(u,v)\xi(u)\eta(v)dudv \;\;\text{and}\;\; \left\|\xi\right\|_t^2 = \langle \xi ,\xi \rangle_t.$$
Note that, for any $t\in[0,T]$, $\langle \xi ,\eta \rangle_t$ is a Hilbert scalar product. Let $\mathcal{H}$ be the completion of the set of continuous functions under this Hilbert norm $\left\|\cdot\right\|_t$ and  $(\xi_n)_n$ be a sequence in $\mathcal{H}$ such that 
$\left\langle \xi_i,\xi_j \right\rangle_T =\delta_{ij}$.
Let $\mathscr{P}_T^H$ be the set of all polynomials of fractional Brownian motion. Namely,  $\mathscr{P}_T^H$  contains all elements  of the form
$$F(\omega)=f\left(\int_0^T\xi_1(t)dB_t^H, \int_0^T\xi_2(t)dB_t^H,\dots, \int_0^T\xi_n(t)dB_t^H \right)$$
where $f$ is a  polynomial function of $n$ variables. The Malliavin derivative  $D_t^H$ of  $F$ is given by 
$$D_s^HF=\sum_{i=1}^n\frac{\partial f}{\partial x_i}\left(\int_0^T\xi_1(t)dB_t^H, \int_0^T\xi_2(t)dB_t^H,\dots, \int_0^T\xi_n(t)dB_t^H \right)\xi_i(s) \quad 0\leq s\leq T.$$ 

Similarly, we can define the Malliavin derivative  $D_tG$ of the Brownian functional 
$$G(\omega)=f\left(\int_0^T\xi_1(t)dB_t, \int_0^T\xi_2(t)dB_t,\dots, \int_0^T\xi_n(t)dB_t \right).$$

The  divergence  operator  $D^{H}$ is closable from $L^{2}(\O,  {\cal F}, \P)$ to  $L^{2}(\O,  {\cal F},\P, {\cal H})$. 
Hence  we   can  consider the space  $ \mathbb{D}_{1, 2}$  is  the completion  of   ${\mathscr P}_{T}^H$  with  the  norm   
$$  || F ||_{1,2}^2 =  \E|F|^{2}  +   \E|| D_{s}^{H} F ||_{T}^2.$$
Now we introduce the Malliavin $\rho$-derivative $\mathbb{D}_t^H$ of  $F$ by
$$\mathbb{D}_t^HF=\int_0^T\rho(t,s)D_s^HFds$$

and  denote  by $\mathbb{L}^{1,2}_H$ the space of all stochastic processes
 $F:\left(\Omega, \mathcal{F}, \Prb\right)\longrightarrow\mathscr{H}$ such that 
$$\E\left(\left\|F\right\|_T^2 +\int_0^T\int_0^T |\mathbb{D}_s^H F_t|^2 ds dt\right) <+\infty.$$
We have the following (see[\cite{Hu}, Proposition 6.25]): 
\begin{theorem}\label{theo-1}
Let $F:\left(\Omega, \mathcal{F}, \P\right)\longrightarrow\mathcal{H}$ be a stochastic processes such that 
$$\E\left(\left\|F\right\|_T^2 +\int_0^T\int_0^T |\mathbb{D}_s^H F_t|^2 ds dt\right) <+\infty.$$
Then, the It\^o-Skorohod type stochastic integral denoted by  $\int_0^TF_sdB_s^H$ exists in  $L^2\left(\Omega, \mathcal{F}, \P\right)$ and satisfies
$$\E\left( \int_0^TF_sdB_s^H\right)=0    \quad\text{and}\quad 
\E \left( \int_0^TF_sdB_s^H\right)^2= \E\left(\Vert F\Vert_T^2 + \int_0^T\int_0^T\mathbb{D}_s^H F_t\mathbb{D}_t^H F_sdsdt \right).$$
\end{theorem}
Let us recall the fractional It\^o formula (see[\cite{Fei}, Theorem 3.1]). 
\begin{theorem} Let $\s_1\in L^2([0,T])$ and $\s_2\in\mathcal{H}$ be deterministic continuous functions.\\ Assume that $\left\|\s_2\right\|_t$ is continuously differentiable as a function of $t\in [0,T]$. Denote
$$X_t= X_0 + \int_0^t\alpha (s)ds + \int_0^t\s_1(s)dB_s + \int_0^t\s_2(s)dB_s^H,$$
where $X_0$ is a constant, $\alpha(t)$ is a deterministic function with $\int_0^t|\alpha(s)|ds<+\infty$. Let $F(t,x)$ be continuously differentiable  with respect to $t$ and twice continuously differentiable with respect to $x$. Then
\begin{align*}
 F(t,X_t)&= F(0,X_0) + \int_0^t\frac{\partial F}{\partial s}(s,X_s)ds +  \int_0^t\frac{\partial F}{\partial x}(s,X_s)dX_s \\
         &+\frac{1}{2}\int_0^t\frac{\partial^2 F}{\partial x^2}(s,X_s)\left[\s_1^2(s) + \frac{d}{ds}\left\|\s_2\right\|_s^2\right]ds,\quad 0\leq t\leq T.
\end{align*}
\end{theorem}

Let us finish this section by giving a fractional It\^o chain rule (see[\cite{Fei}, Theorem 3.2]). 
\begin{theorem} Assume that for $i=1,2$, the processes $\mu_i$, $\alpha_i$ and $\vartheta_i$, satisfy
$$\E\left[\int_0^T\mu_i^2(s)ds + \int_0^T\alpha_i^2(s)ds + \int_0^T\vartheta_i^2(s)ds\right]<\infty.$$
Suppose that $D_t\alpha_i(s)$ and $\mathbb{D}^H_t\vartheta_i(s)$ are continuously differentiable with respect to $(s,t)\in [0,T]^2$ for almost all $\omega\in\Omega$. Let $X_t$ and $Y_t$ be two processes satisfying
\begin{align*}
 X_t &= X_0 + \int_0^t\mu_1(s)ds +  \int_0^t\alpha_1(s)dB_s  + \int_0^t\vartheta_1(s)dB_s^H, \quad\quad 0\leq t\leq T,\\
 Y_t &= Y_0 + \int_0^t\mu_2(s)ds +  \int_0^t\alpha_2(s)dB_s  + \int_0^t\vartheta_2(s)dB_s^H, \quad\quad 0\leq t\leq T.
\end{align*}
If for $i=1,2$, the following conditions hold:
$$\E\left[\int_0^T|D_t\alpha_i(s)|^2dsdt\right]<+\infty ,\quad\quad \E\left[\int_0^T|\mathbb{D}_t^H\vartheta_i(s)|^2dsdt\right]<+\infty,$$
then
\begin{align*}
 X_tY_t &= X_0Y_0 +\int_0^tX_sdY_s + \int_0^tY_sdX_s   \\
       &+\int_0^t\left[\alpha_1(s)D_sY_s + \alpha_2(s)D_sX_s  + \vartheta_1(s)\mathbb{D}_s^HY_s + \vartheta_2(s)\mathbb{D}_s^HX_s \right]ds,
\end{align*}
which may be written formally as
$$d\left(X_tY_t\right) = X_tdY_t + Y_tdX_t +\left[\alpha_1(t)D_tY_t + \alpha_2(t)D_tX_t  + \vartheta_1(t)\mathbb{D}_t^HY_t + \vartheta_2(t)\mathbb{D}_t^HX_t \right]dt.$$
\end{theorem}
In order to present a stochastic averaging principle, we need the following \cite[Lemma 1]{Xu14}
\begin{lemma}\label{var}
Let $B^H_t$ be a fractional Brownian motion with $\frac{1}{2}< H <1$, and $u(s)$ be a stochastic process in $\mathbb{L}^{1,2}_H$. For every $T <+\infty$, there exists
a constant $C_0(H, T )=HT^{2H-1}$ such that
$$\E\left[\left(\int_0^T\left|u(s)\right|dB_s^H\right)^2\right] \leq C_0(H,T)\E\left[\int_0^T\left|u(s)\right|^2ds\right] +C_0T^2.$$
\end{lemma}

We are now in position to move on to study our main subject.
\section{Averaging Principle for SFrBSDEs}
\subsection{SFrBSDEs}
Let us consider the following process
$$\eta_t= \e_0 + b(t) + \int_0^t\sigma_1(s)dB_s +\int_0^t\sigma_2(s)dB_s^H,   \quad   0\le t\le T$$
where the coefficients  $\e_0$, $b$, $\s_1$ and $\s_2$ satisfy:
\begin{itemize}
\item $\e_0$ is a given constant,
\item	$b,\s_1, \s_2: [0,T]\rightarrow\mathbf{R}$ are  deterministic continuous functions, $\s_1$ and $\s_2$ are differentiable and $\s_1(t)\neq 0,\;\s_2(t)\neq 0$ such that
\begin{equation}
|\sigma|_t^2=\int_0^t\sigma_1^2(s)ds + \left\|\sigma_2\right\|_t^2,   \quad   0\le t\le T,\label{qasi1}
\end{equation}
\end{itemize} 
$$\text{where} \quad\quad\quad\left\|\s_2\right\|_{t}^2=H(2H-1)\int_0^t\int_0^t |u-v|^{2H-2}\s_2(u)\s_2(v)dudv.\quad\quad\quad\quad\quad\quad\quad\quad\quad\quad$$
$$\text{Let} \quad\quad\quad\quad \hat{\s}_2(t)=\int_0^t\rho(t,v)\s_2(v)dv,   \quad   0\le t\le T.\quad\quad\quad\quad\quad\quad\quad\quad\quad\quad\quad\quad\quad\quad\quad\quad\quad\quad$$

The next Remark will be useful in the sequel.
\begin{remark}\label{rm1} The function $|\s|_t^2$ defined by eq.\eqref{qasi1} is continuously differentiable with respect to $t$ on $[0,T]$, and 
\begin{itemize}
\item[a)] $\frac{d}{dt}|\s|_t^2 =\s_1^2(t)+ \frac{d}{dt}\left\|\s_2\right\|_t^2= \s_1^2(t) + \s_2(t)\hat{\s}_2(t)>0, \quad 0\le t\le T.$ 
\item[b)] for a suitable constant  $C_1>0$, $\;\inf_{0\le t\le T} \frac{\hat \s_2(t)}{\s_2(t)}\geq C_1.$ 
\end{itemize}
\end{remark}
Given $\xi$ a  measurable  real   valued random  variable  and the   function
$$ f : \O\times [0, T]\times   \R\times \R\times\R\times \R \to  \R,$$  
we  consider the  BSDEs driven by  both standard and fractional Brownian motion (FrBSDEs) 
\begin{equation}\label{baw}
Y_t = \xi + \!\int_t^Tf(s, \e_s, Y_s, Z_{1,s},Z_{2,s})ds -\!\int_t^TZ_{1,s}dB_s - \!\int_t^TZ_{2,s}dB_s^H,\quad 0\leq t\leq T.
\end{equation}
We  introduce the following sets (where $\E$ denotes the mathematical expectation with respect to the probability measure $\Prb$) :
\begin{itemize}
\item $\mathscr{C}_{\mbox{pol}}^{1,2}\!\left([0, T]\times \R\right)$  is the space of all $\mathscr{C}^{1, 2}$-functions over 
$[0, T]\times \R$, which together with their derivatives are of polynomial growth, 
\item ${\cal V}_{[0,T]}=\left\{Y=\psi(\cdot,\e):\; \psi\in\mathscr{C}_{\mbox{pol}}^{1, 2}([0, T]\times\R), \frac{\partial \psi}{\partial t} \; \text{is bounded},\; t\in[0, T] \right\},$
\item $\tl{\cal  V}_{[0,T]}$ the  completion  of  ${\cal V}_{[0,T]}$   under  the  following  norm  
$$\Vert Y\Vert =\left(\int_0^T \E|Y_t|^2dt\right)^{1/2}=\left(\int_0^T\E|\psi(t,\e_t)|^2dt \right)^{1/2}.$$ 
\end{itemize}

\begin{definition}
A triplet of  processes  $(Y_t, Z_{1,t}, Z_{2,t})_{0\le t \le T}$ is called  a solution to SFrBSDE \eqref{baw},
if $(Y_t, Z_{1,t}, Z_{2,t})_{0\le t \le T}\in\tl{\cal  V}_{[0,T]}\times\tl{\cal  V}_{[0,T]}\times\tl{\cal  V}_{[0,T]}$ and satisfies eq.\eqref{baw}. 
\end{definition}

We have the following (see [\cite{Fei}, Theorem 5.3])
\begin{theorem} \label{eth0}
Assume that $\s_1$ and $\s_2$ are continuous and $|\s|_t^2$ defined by eq.\eqref{qasi1} is a strictly increasing function of $t$.
Let the SFrBSDE \eqref{baw} has a  solution  of the form\\ $ \left(Y_t=\psi(t,\eta_t),\;  Z_{1,t}=-\varphi_1(t,\eta_t), \;Z_{2,t}=-\varphi_2(t,\eta_t)\right)$, where $\psi\in\mathscr{C}^{1, 2}([0, T]\times\R)$. Then
$$ \varphi_1(t,x)=\s_1(t)\psi_x^\prime(t,x), \quad \varphi_2(t,x)=\s_2(t)\psi_x^\prime(t,x).$$
\end{theorem}
The next proposition will be useful in the sequel. 
\begin{proposition} \label{pro} Let $(Y_t, Z_{1,t}, Z_{2,t})_{0\le t \le T}$ be a solution of the SFrBSDE \eqref{baw}. Then  for almost $t\in[0, T]$,  
$$D_tY_t= Z_{1,t},\quad\text{and} \quad  \mathbb{D}^H_tY_t=\frac{\hat \s_2(t)}{\s_2(t)}Z_{2,t}.$$  
\end{proposition}
\begin{proof} 
Since $(Y_t, Z_{1,t}, Z_{2,t})$ satisfies the SFrBSDE \eqref{baw} then we  have  $Y = \psi(\cdot,\e)$  where\\ $\psi\in \mathscr{C}^{1, 2}([0, T]\times \R)$. From Theorem \ref{eth0}, we have
$$Z_{1,t}=\s_1(t)\psi_x^\prime(t,x), \quad Z_{2,t}=\s_2(t)\psi_x^\prime(t,x).$$
Then   we  can  write  $D_tY_t = \s_1(t)\psi_x^\prime(t,x)= Z_{1,t}$  and
\begin{align*}
\mathbb{D}_t^{H}Y_t =\int_{0}^{T} \phi(t,s)D_s^{H} \psi(t, \e_{t}) ds &=   \psi^{\prime}_{x} (t, \e_{t})  \int_{0}^{T} \phi(t,s) \s_2(s)ds \\ &=  \widehat\s_2 (t) \psi^{\prime}_{x} (t, \e_{t})= \frac{ \widehat\s_2(t)}{\s_2(t)}Z_{2,t}.
\end{align*}
\end{proof}

\subsection{An averaging principle}
In this section, we are going to investigate the averaging principle for the FrBSDEs under non Lipschitz coefficients. Let us consider the standard form of equation
(\ref{baw}):
\begin{equation}\label{per}
Y^{\varepsilon}_{t}=\xi+\varepsilon^{2H}\int_{t}^{T}f\left(r, \eta^{\varepsilon}_r,Y^{\varepsilon}_{r},Z^{\varepsilon}_{1,r},Z^{\varepsilon}_{2,r}\right)dr- \varepsilon^{H}\int_{t}^{T}Z^{\varepsilon}_{1,r}dB_{r} 
- \varepsilon^H\int_{t}^{T}Z^{\varepsilon}_{2,r}dB^H_{r}, \quad t\in [0, T];
\end{equation}
where  $\eta^{\varepsilon}_{t} = \eta_{0} + \varepsilon^{2H}\disp\int_0^t b(s)ds+ \varepsilon^{H}\int_0^t\sigma_1(s)dB_s + \varepsilon^{H}\int_0^t\sigma_2(s)dB_s^H, \quad t\in [0, T].$

According to the second part, equation (\ref{per}) also has an adapted unique and square integrable solution. We will examine whether the solution $Y^{\varepsilon}_{t}$ can be approximated to the solution process $\overline{Y}_{t}$ of the simplified equation:
\begin{equation}\label{noper}
\overline{Y}_{t}=\xi+\varepsilon^{2H}\int_{t}^{T}\overline{f}\left(\eta^{\varepsilon}_r,\overline{Y}_{r},\overline{Z}_{1,r},\overline{Z}_{2,r}\right)dr- \varepsilon^{H}\int_{t}^{T}\overline{Z}_{1,r}dB_{r}- \varepsilon^{H}\int_{t}^{T}\overline{Z}_{r}dB^H_{r}, \quad t\in [0, T];
\end{equation}
where    $\left(\overline{Y}_{t},\overline{Z}_{1,t},\overline{Z}_{2,t}\right)$ has the same properties as $\left(Y^{\varepsilon}_{t},Z^{\varepsilon}_{1,t},Z^{\varepsilon}_{2,t}\right)$.\\

We assume that the coefficients $f$ and $\overline{f}$ are continuous functions and satisfy the following assumption:
\begin{itemize}
\item[$\bullet\;$\bf (A1)] There exists $L>0$  such that, for all $\left(t,x,y,z_{1},z_{2},y',z'_{1},z'_{2}\right) \in \left[0,T\right]\times\R^7$, we have
	\begin{align*}
	\left|f(t,x,y,z_{1},z_{2})-f(t,x,y',z'_{1},z'_{2})\right|^{2}&\leq L\left(\left|y-y'\right|^{2}+\left|z_{1}-z'_{1}\right|^{2}+\left|z_{2}-z'_{2}\right|^{2}\right) 
	\end{align*}
\item[$\bullet\;$\bf (A2)] For any $t\in \left[0,T_{1}\right] \subset \left[0,T\right]$ and for all $\left(x,y,z_1,z_2\right)\in\R\times\R\times\R\times\R$, we have
$$\dfrac{1}{T_{1}-t}\int_{t}^{T_1}\left|f(s,x,y,z_1,z_2)-\overline{f}(x,y,z_1,z_2)\right|^{2}ds \leq \phi(T_{1}-t)\left(1+\left|y\right|^{2}+\left|z_1\right|^{2}+\left|z_2\right|^{2}\right) $$
where $\phi$ is a bounded function.
\end{itemize}

In what follows, we establish the result which will be useful in the sequel.
\begin{lemma}\label{1} Suppose that the original SFrBSDEs \eqref{per} and the averaged SFrBSDEs \eqref{noper}
both satisfy the assumptions $\bf(A1)$ and $\bf(A2)$. For a given arbitrarily small number $u\in\left[0,t\right]\subset\left[0,T\right]$, there exist $L_1>0$ and $C_2>0$ such that
\begin{equation}
	\E\left[\int_{u}^{T}\left[\left|Z^{\varepsilon}_{1,s}-\overline{Z}_{1,s}\right|^{2}+\left|Z^{\varepsilon}_{2,s}-\overline{Z}_{2,s}\right|^{2}\right]ds \right]\leq L_1\E\left[\int_{u}^{T}\left|Y^{\varepsilon}_{s}-\overline{Y}_{s}\right|^{2}ds\right] + C_2\left(T-u\right).
	\end{equation} 
\end{lemma}
\begin{proof}
Let us  define $\overline{\Delta\d}^{\varepsilon} = \d^{\varepsilon}-\overline{\d}$ for  a  process  $\d\in \{Y, Z_1,Z_2\}$.  

It is easily seen  that  the  pair  of processes  $\left(\overline{\Delta Y}^{\varepsilon}_t,\overline{\Delta Z}^{\varepsilon}_{1,t},\overline{\Delta Z}^{\varepsilon}_{2,t}\right)_{{0\le t \le T}}$ solves the SFrBSDE   
 \begin{eqnarray*}
\overline{\Delta Y}^{\varepsilon}_t &= \varepsilon^{2H}\!\int_t^T\left(f(s,\eta^{\varepsilon}_s, Y^{\varepsilon}_{s},Z^{\varepsilon}_{1,s},Z^{\varepsilon}_{2,s})-\overline{f}(\eta^{\varepsilon}_s,\overline{Y}_{s},\overline{Z}_{1,s},\overline{Z}_{2,s})\right) ds -\varepsilon^{H}\!\int_t^T\overline{\Delta Z}^{\varepsilon}_{1,s}dB_s \\ 
&- \varepsilon^{H}\!\int_t^T\overline{\Delta Z}^{\varepsilon}_{2,s}dB^H_s.
\end{eqnarray*} 

Applying It\^o's formula to $ \left|\overline{\Delta Y}^{\varepsilon}_t \right|^{2}$, we obtain
	\begin{align}\label{ito-1}
\left|\overline{\Delta Y}^{\varepsilon}_t \right|^{2}&+\varepsilon^H\int_{u}^{T}D_s\overline{\Delta Y}^{\varepsilon}_s\overline{\Delta Z}^{\varepsilon}_{1,s}ds
	+\varepsilon^H\int_{u}^{T}\mathbb{D}^H_s\overline{\Delta Y}^{\varepsilon}_s\overline{\Delta Z}^{\varepsilon}_{2,s}ds\nonumber\\
	=&2\varepsilon^{2H}\int_{u}^{T}\overline{\Delta Y}^{\varepsilon}_s\left(f(s,\eta^{\varepsilon}_s, Y^{\varepsilon}_{s},Z^{\varepsilon}_{1,s},Z^{\varepsilon}_{2,s})-\overline{f}(\eta^{\varepsilon}_s,\overline{Y}_{s},\overline{Z}_{1,s},\overline{Z}_{2,s})\right) ds\nonumber\\
&- 2\varepsilon^H\int_{u}^{T}\overline{\Delta Y}^{\varepsilon}_s\overline{\Delta Z}^{\varepsilon}_{1,s}dB_s
- 2\varepsilon^H\int_{u}^{T}\overline{\Delta Y}^{\varepsilon}_s\overline{\Delta Z}^{\varepsilon}_{2,s}dB^H_s
\end{align}	
Using  the  fact  that  $\left(\overline{\Delta Y}^{\varepsilon}_s,\overline{\Delta Z}^{\varepsilon}_{1,s},\overline{\Delta Z}^{\varepsilon}_{2,s}\right)_{t\le s \le T}  \in   \widetilde{\cal  V}_{[0, T]}\times  \widetilde{\cal  V}_{[0, T]}\times  \widetilde{\cal  V}_{[0, T]}$  and   $ {\cal  V}_{[0, T]} \subset  \mathbb{L}^{1,2}_{H}$ (see Lemma 8  in  \cite{Mat-Nie})   which  implies  in  fact $ F_{i,s}  = \overline{\Delta Y}^{\varepsilon}_s\overline{\Delta Z}^{\varepsilon}_{i,s}\in\mathbb{L}_{H}^{1,2},\,$ (where $i=1,2$). Then  by Theorem \ref{theo-1}, we   have  
 $$\E\left[\int_0^T\overline{\Delta Y}^{\varepsilon}_s\overline{\Delta Z}^{\varepsilon}_{1,s}d B_s+\int_0^T\overline{\Delta Y}^{\varepsilon}_s\overline{\Delta Z}^{\varepsilon}_{2,s}d B^H_s\right]= 0$$
Hence   we deduce from  \eqref{ito-1}
\begin{align}\label{H1}
\E\!\left[\left|\overline{\Delta Y}^{\varepsilon}_t \right|^{2}\right]&+\varepsilon^H\E\!\left[\int_{u}^{T}D_s\overline{\Delta Y}^{\varepsilon}_s\overline{\Delta Z}^{\varepsilon}_{1,s}ds\right]
	+\varepsilon^H\E\!\left[\int_{u}^{T}\mathbb{D}^H_s\overline{\Delta Y}^{\varepsilon}_s\overline{\Delta Z}^{\varepsilon}_{2,s}ds\right]\nonumber\\
	&=2\varepsilon^{2H}\E\left[\int_{u}^{T}\overline{\Delta Y}^{\varepsilon}_s\left(f(s,\eta^{\varepsilon}_s, Y^{\varepsilon}_{s},Z^{\varepsilon}_{1,s},Z^{\varepsilon}_{2,s})-\overline{f}(\eta^{\varepsilon}_s,\overline{Y}_{s},\overline{Z}_{1,s},\overline{Z}_{2,s})\right)ds\right] \nonumber\\
&\leq 2\varepsilon^{2H}\E\left[\int_{u}^{T}\overline{\Delta Y}^{\varepsilon}_s\left(f(s,\eta^{\varepsilon}_s, Y^{\varepsilon}_{s},Z^{\varepsilon}_{1,s},Z^{\varepsilon}_{2,s})-f(s,\eta^{\varepsilon}_s,\overline{Y}_{s},\overline{Z}_{1,s},\overline{Z}_{2,s})\right)ds\right] \nonumber\\
&\quad+ 2\varepsilon^{2H}\E\left[\int_{u}^{T}\overline{\Delta Y}^{\varepsilon}_s\left(f(s,\eta^{\varepsilon}_s,\overline{Y}_{s},\overline{Z}_{1,s},\overline{Z}_{2,s})-\overline{f}(\eta^{\varepsilon}_s,\overline{Y}_{s},\overline{Z}_{1,s},\overline{Z}_{2,s})\right)ds\right]\\
	&= E_{1}+E_{2} \nonumber,
\end{align}	
where $E_{1}=\disp 2\varepsilon^{2H}\E\left[\int_{u}^{T}\overline{\Delta Y}^{\varepsilon}_s\left(f(s,\eta^{\varepsilon}_s, Y^{\varepsilon}_{s}, Z^{\varepsilon}_{1,s}, Z^{\varepsilon}_{2,s})-f(s,\eta^{\varepsilon}_s,\overline{Y}_{s},\overline{Z}_{1,s},\overline{Z}_{2,s})\right)ds\right] $ \\and
$E_{2}=2\varepsilon^{2H}\E\left[\int_{u}^{T}\overline{\Delta Y}^{\varepsilon}_s\left(f(s,\eta^{\varepsilon}_s, \overline{Y}_{s},\overline{Z}_{1,s},\overline{Z}_{2,s})-\overline{f}(\eta^{\varepsilon}_s,\overline{Y}_{s},\overline{Z}_{1,s},\overline{Z}_{2,s})\right)ds\right] $

For $E_{1}$, by using the condition $(\bf A1)$ and Holder's inequality, for any $\alpha > 0$,\\ $2ab\leq\alpha a^{2} + b^{2}/\alpha$, we deduce that
	\begin{align}\label{H2}
	E_{1} \leq &\alpha\varepsilon^{2H}\E\left[\int_{u}^{T}\left|\overline{\Delta Y}^{\varepsilon}_s\right|^{2}ds\right]+\dfrac{\varepsilon^{2H}}{\alpha} \E\left[\int_{u}^{T}\left|f(s,\eta^{\varepsilon}_s, Y^{\varepsilon}_{s},Z^{\varepsilon}_{1,s},Z^{\varepsilon}_{2,s})-f(s,\eta^{\varepsilon}_s,\overline{Y}_{s},\overline{Z}_{1,s},\overline{Z}_{2,s})\right|^{2}ds\right] \nonumber\\
\leq &\varepsilon^{2H}\left(\alpha+\dfrac{L}{\alpha}\right)\E\left[\int_{u}^{T}\left|\overline{\Delta Y}^{\varepsilon}_s\right|^{2}ds\right] +\dfrac{L\varepsilon^{2H}}{\alpha} \E\left[\int_{u}^{T}\left[\left|\overline{\Delta Z}^{\varepsilon}_{1,s}\right|^{2}+\left|\overline{\Delta Z}^{\varepsilon}_{2,s}\right|^{2}\right]ds\right]\nonumber\\
\end{align}
	
For $E_{2}$, by using  assumption $\bf (A2)$,  Holder's inequality and Young's inequality, we have
\begin{align}\label{H3}
&E_{2}\leq 2\varepsilon^{2H}\E\left[\!\left(\int_{u}^{T}\left|\overline{\Delta Y}^{\varepsilon}_s\right|^{2}ds\!\right)^{\frac{1}{2}} \left(\int_{t}^{T}\left|f(s,\eta^{\varepsilon}_s,\overline{Y}_{s},\overline{Z}_{1,s},\overline{Z}_{2,s})-\overline{f}(\eta^{\varepsilon}_s,\overline{Y}_{s},\overline{Z}_{1,s},\overline{Z}_{2,s})\right|^{2}ds\!\right)^{\frac{1}{2}}\!\right]\nonumber\\
\leq& 2\varepsilon^{2H}\E\left[\!\left(\left(T\!-\!u\right)\!\int_{u}^{T}\left|\overline{\Delta Y}^{\varepsilon}_s\right|^{2}ds\!\right)^{\frac{1}{2}}\left(\dfrac{1}{T\!-\!u}\!\int_{u}^{T}\left|f(s,\eta^{\varepsilon}_s,\overline{Y}_{s},\overline{Z}_{1,s},\overline{Z}_{2,s})-\overline{f}(\eta^{\varepsilon}_s,\overline{Y}_{s},\overline{Z}_{1,s},\overline{Z}_{2,s})\right|^{2}ds\!\right)^{\frac{1}{2}}\!\right] \nonumber\\
&\leq  2\varepsilon^{2H}C_2\E\left[\!\left(\!\int_{u}^{T}\left|\overline{\Delta Y}^{\varepsilon}_s\right|^{2}ds\!\right)^{\frac{1}{2}}\right]\nonumber\\
&\leq  \varepsilon^{2H}C_2\E\left[\int_{u}^{T}\left|\overline{\Delta Y}^{\varepsilon}_s\right|^{2}ds+ T-u\right]\nonumber\\
&\leq  \varepsilon^{2H}C_2\E\left[\int_{u}^{T}\left|\overline{\Delta Y}^{\varepsilon}_s\right|^{2}ds\right] + \varepsilon^{2H}C_2\left(T-u\right);
\end{align}
where  $C_2=\!\disp\sqrt{\left(T\!-\!u\right)\sup_{u\leq s \leq T}\!\phi(s\!-\!u)\left[1+\sup_{u\leq s \leq T}\E(\left|\overline{Y}_{s}\right|^{2})+\sup_{u\leq s \leq T}\!\E(\left|\overline{Z}_{1,s}\right|^{2})+\sup_{u\leq s \leq T}\!\E(\left|\overline{Z}_{2,s}\right|^{2})\right]}.$\\

By the stochastic representation given in Proposition \ref{pro} and the Remark \ref{rm1}, we have
$$\E\!\left[\int_{u}^{T}\!D_s\overline{\Delta Y}^{\varepsilon}_s\overline{\Delta Z}^{\varepsilon}_{1,s}\!ds\right]\!=\!\E\!\left[\int_{u}^{T}\!\left|\overline{\Delta Z}^{\varepsilon}_{1,s}\right|^2\!ds\right]\;\text{and}\; \E\!\left[\!\int_{u}^{T}\!\mathbb{D}^H_s\overline{\Delta Y}^{\varepsilon}_s\overline{\Delta Z}^{\varepsilon}_{2,s}\!ds\right]\!\geq\! C_1\E\!\left[\int_{u}^{T}\!\left|\overline{\Delta Z}^{\varepsilon}_{2,s}\right|^2\!ds\right]$$

Putting pieces together, we deduce from \eqref{H1} that
\begin{align}\label{HH4}
\E\!\left[\left|\overline{\Delta Y}^{\varepsilon}_t \right|^{2}\right]&+\varepsilon^H\E\!\left[\!\int_{u}^{T}\left|\overline{\Delta Z}^{\varepsilon}_{1,s}\right|^2ds\right]	+C_1\varepsilon^H\E\!\left[\!\int_{u}^{T}\left|\overline{\Delta Z}^{\varepsilon}_{2,s}\right|^2ds\right]\nonumber\\
&\leq \varepsilon^{2H}\left(\alpha+\dfrac{L}{\alpha}+C_2\right)\E\left[\int_{u}^{T}\left|\overline{\Delta Y}^{\varepsilon}_s\right|^{2}ds\right]+ \varepsilon^{2H}C_2\left(T-u\right) \nonumber\\
 &\quad+\dfrac{L\varepsilon^{2H}}{\alpha} \E\left[\int_{u}^{T}\left[\left|\overline{\Delta Z}^{\varepsilon}_{1,s}\right|^{2}+\left|\overline{\Delta Z}^{\varepsilon}_{2,s}\right|^{2}\right]ds\right]
\end{align}

Hene if we choose $\alpha\!=\!\alpha_0$ satisfying $\disp \frac{\varepsilon^H}{\alpha_0}\min\left\{\alpha_0-L\varepsilon^{H}, \alpha_0C_1-L\varepsilon^H\right\}=\varepsilon^{2H}$, then we obtain

\begin{align*}
\varepsilon^{2H}\E\!\left[\int_{u}^{T}\left[\left|\overline{\Delta Z}^{\varepsilon}_{1,s}\right|^2+\left|\overline{\Delta Z}^{\varepsilon}_{2,s}\right|^2\right]ds\right]\!&\leq \varepsilon^{2H}\left(\alpha_0+\dfrac{L}{\alpha_0}+C_2\right)\E\left[\int_{u}^{T}\left|\overline{\Delta Y}^{\varepsilon}_s\right|^{2}ds\right] + \varepsilon^{2H}C_2\left(T\!-\!u\right) 
\end{align*}	

Thus,
\begin{align*}
\E\!\left[\int_u^T\left[\left|Z^{\varepsilon}_{1,s}-\overline{Z}_{1,s}\right|^2+\left|Z^{\varepsilon}_{2,s}-\overline{Z}_{2,s}\right|^2\right]ds\right]\leq L_1 \E\int_u^T\left|Y^{\varepsilon}_{s}-\overline{Y}_{s}\right|^{2}ds+ C_2(T-u),
\end{align*}
where $L_1= \alpha_0+\dfrac{L}{\alpha_0}+C_2$.  This completes the proof.
\end{proof}

Now, we claim the main theorem showing the relationship between solution processes $Y^{\varepsilon}_{t}$ to the original \eqref{per} and $\overline{Y}_{t}$ to the averaged \eqref{noper}.  It shows that the solution of the averaged \eqref{noper}  converges to that of the original \eqref{per}  in mean square sense.
\begin{theorem}\label{t1}
Under the assumption of Lemma \ref{1} are satisfied. For a given arbitrarily small number $\delta_{1}>0$, there exists $\varepsilon_{1}\in \left[0,\varepsilon_{0}\right]$ and $\beta \in \left[0,1\right]$ such that for all $\varepsilon\in \left[0,\varepsilon_{1}\right]$ having
\begin{displaymath}
\sup_{T\varepsilon^{1-\beta}\leq t \leq T}\E\left|Y^{\varepsilon}_{t}-\overline{Y}_{t}\right|^{2} \leq \delta_{1} . 
\end{displaymath}
\end{theorem}
\begin{proof}
	With the help of Lemma \ref{1}, now we can prove the Theorem \ref{t1}. Using the elementary inequelity and the isometry property, we derive that
	\begin{align}\label{ARAKHIM1}
\E\left[\left|\overline{\Delta Y}^{\varepsilon}_s\right|^{2}\right]\leq&2\varepsilon^{4H}\E\left[\left|\int_{u}^{T}\left[f(s,\eta^{\varepsilon}_{s},Y^{\varepsilon}_{s},Z^{\varepsilon}_{1,s},Z^{\varepsilon}_{2,s})\!-\!\overline{f}(\eta^{\varepsilon}_{s},\overline{Y}_{s},\overline{Z}_{1,s},\overline{Z}_{2,s})\right]ds\right|^{2}\right]\nonumber\\
&+2\E\left[\left|\varepsilon^{H}\!\int_{u}^{T}\overline{\Delta Z}^{\varepsilon}_{1,s}dB_{s}+ \varepsilon^{H}\!\int_{u}^{T}\overline{\Delta Z}^{\varepsilon}_{2,s}dB^H_{s}\right|^{2}\right]\nonumber\\
	&\leq 4\varepsilon^{4H}\E\left[\left|\int_{u}^{T}\left[f(s,\eta^{\varepsilon}_{s},Y^{\varepsilon}_{s},Z^{\varepsilon}_{1,s},Z^{\varepsilon}_{2,s})-f(s,\eta^{\varepsilon}_{s},\overline{Y}_{s},\overline{Z}_{1,s},\overline{Z}_{2,s})\right]ds\right|^{2}\right]\nonumber\\
&+4\varepsilon^{4H}\E\left[\left|\int_{u}^{T}\left[f(s,\eta^{\varepsilon}_{s},\overline{Y}_{s},\overline{Z}_{1,s},\overline{Z}_{2,s})-\overline{f}(\eta^{\varepsilon}_{s},\overline{Y}_{s},\overline{Z}_{1,s},\overline{Z}_{2,s})\right]ds\right|^{2}\right]\nonumber\\
&+4\varepsilon^{4H}\E\left[\left|\int_{u}^{T}\overline{\Delta Z}^{\varepsilon}_{1,s}dB_{s}\right|^{2}\right] +4\varepsilon^{4H}\E\left[\left|\int_{u}^{T}\overline{\Delta Z}^{\varepsilon}_{2,s}dB^H_{s}\right|^{2}\right]\nonumber\\
&=I_{1}+I_{2}+I_{3}+I_4
	\end{align}
Applying Holder's inequality and the assumption $\bf (A1)$, we obtain
\begin{align}\label{ARAKHIM2}
I_{1}&\leq  4(T-u)\varepsilon^{4H}\E\left[\int_{u}^{T}\left|f(s,\eta^{\varepsilon}_{s},Y^{\varepsilon}_{s},Z^{\varepsilon}_{1,s},Z^{\varepsilon}_{2,s})-f(s,\eta^{\varepsilon}_{s},\overline{Y}_{s},\overline{Z}_{1,s},\overline{Z}_{2,s})\right|^{2}ds\right] \nonumber\\
&\leq  4(T-u)L\varepsilon^{4H}\E\left[\int_{u}^{T}\left[\left|\overline{\Delta Y}^{\varepsilon}_s\right|^{2}+ \left|\overline{\Delta Z}^{\varepsilon}_{1,s}\right|^{2}+\left|\overline{\Delta Z}^{\varepsilon}_{2,s}\right|^{2}\right]ds\right]
\end{align}

Then, together with Holder's inequality and the assumption $\bf (A2)$, we get
\begin{align}\label{ARAKHIM3}
I_{2}\leq& 4(T-u)\varepsilon^{4H}\E\left[\int_{u}^{T}\left|f(s,\eta^{\varepsilon}_{s},\overline{Y}_{s}, \overline{Z}_{1,s}, \overline{Z}_{2,s}) -\overline{f}(\eta^{\varepsilon}_{s},\overline{Y}_{s},\overline{Z}_{1,s},\overline{Z}_{2,s})\right|^{2}ds\right]\nonumber\\
\leq & 4(T-u)^{2}\varepsilon^{4H}\E\left[\dfrac{1}{T-u}\int_{u}^{T}\left|f(s,\eta^{\varepsilon}_{s},\overline{Y}_{s},\overline{Z}_{1,s},\overline{Z}_{2,s}) -\overline{f}(\eta^{\varepsilon}_{s},\overline{Y}_{s},\overline{Z}_{1,s},\overline{Z}_{2,s})\right|^{2}ds\right]\nonumber\\
\leq &C_{3}(T-u)^{2}\varepsilon^{4H},
\end{align}
where $C_{3}=4\disp\sup_{u\leq s\leq T}\left[\phi(s\!-\!u)\right]\left(1+\sup_{u\leq s \leq T}\E\left(\left|\overline{Y}_{s}\right|^2\right)+\sup_{u\leq s \leq T}\E\left(\left|\overline{Z}_{1,s}\right|^2\right)+\sup_{u\leq s \leq T}\E\left(\left|\overline{Z}_{2,s}\right|^2\right)\right)$.\\

By the Lemma \ref{var}, we obtain
\begin{align}\label{18}
I_{3}+I_4&\leq2\varepsilon^{2H}HT^{2H-1}\E\left[\int_{u}^{T}\left[\left|\overline{\Delta Z}^{\varepsilon}_{1,s}\right|^{2}+\left|\overline{\Delta Z}^{\varepsilon}_{2,s}\right|^{2}\right]ds\right] + 4\varepsilon^{2H}C_0T^2.
\end{align}

Using above inequalities, from \eqref{ARAKHIM1}, we deduce
\begin{align}
\sup_{u\leq t \leq T}\E\left[\left|\overline{\Delta Y}^{\varepsilon}_t\right|^{2}\right]\leq & \left(4(T-u)L\varepsilon^{4H}+2\varepsilon^{2H}HT^{2H-1}\right) \sup_{u\leq t \leq T}\E\left[\int_{u}^{T}\left[\left|\overline{\Delta Z}^{\varepsilon}_{1,s}\right|^{2}+\left|\overline{\Delta Z}^{\varepsilon}_{2,s}\right|^{2}\right]ds \right] \nonumber\\
&+4(T\!-\!u)L\varepsilon^{4H}\sup_{u\leq t \leq T}\E\int_{u}^{T}\left|\overline{\Delta Y}^{\varepsilon}_s\right|^{2}ds +C_{3}(T-u)^{2}\varepsilon^{4H}+ 4\varepsilon^{2H}C_0T^2 \nonumber
\end{align}
Applying Lemma \ref{1}  to the above inequality we get
\begin{align}
\sup_{u\leq t \leq T}&\E\left[\left|\overline{\Delta Y}^{\varepsilon}_t\right|^{2}\right]\leq  \left[4(T\!-\!u)L\varepsilon^{4H}\left(L_1+1\right)+2L_1\varepsilon^{2H}HT^{2H-1}\right]\int_{u}^{T}\sup_{u\leq s_{1} \leq s}\E\left|\overline{\Delta Y}^{\varepsilon}_{s_1}\right|^{2}ds\nonumber\\
&+\varepsilon^{2H}\left[\left(4(T\!-\!u)L\varepsilon^{2H}+2HT^{2H-1}\right)C_2(T\!-\!u) + C_{3}(T\!-\!u)^{2}\varepsilon^{2H}+ 4C_0T^2\right].
\end{align}

Thanks to Gronwall's inequality, we obtain
\begin{align*}
\sup_{u\leq t \leq T}\!\E\left|\overline{\Delta Y}^{\varepsilon}_{t}\right|^{2}\leq& \varepsilon^{2H}\left[\left(4(T\!-\!u)L\varepsilon^{2H}+2HT^{2H-1}\right)C_2(T\!-\!u) + C_{3}(T\!-\!u)^{2}\varepsilon^{2H}+ 4C_0T^2\right]\\
&\quad\times e^{(T\!-\!u)\left[4(T\!-\!u)L\varepsilon^{4H}\left(L_1+1\right)+2L_1\varepsilon^{2H}HT^{2H-1}\right]}.
\end{align*}
Obviously, the above estimate implies that there exist $\beta\in\left[0,1\right]$  and $K > 0$ such that for evry $t\in(0, K\varepsilon^{-2H\beta}]\subseteq[0,T]$,
\begin{equation}
\sup_{K\varepsilon^{1-\beta}\leq t \leq T}\E\left|Y^{\varepsilon}_{t}-\overline{Y}_{t}\right|^{2}\leq C_{4}\varepsilon^{1-2H\beta},
\end{equation}
in which 
\begin{align*}
C_{4}&=\left[\left(4(T\!-\!K\varepsilon^{-2H\beta})L\varepsilon^{2H}+2HT^{2H-1}\right)C_2(T\!-\!K\varepsilon^{-2H\beta}) + C_{3}(T\!-\!K\varepsilon^{-2H\beta})^{2}\varepsilon^{2H}+ 4C_0T^2\right]\\
&\qquad\times\varepsilon^{2H(1+\beta)-1} e^{(T\!-\!K\varepsilon^{-2H\beta})\left[4(T\!-\!K\varepsilon^{-2H\beta})L\varepsilon^{4H}\left(L_1+1\right)+2L_1\varepsilon^{2H}HT^{2H-1}\right]}
\end{align*}
is constant. 

Consequently, for any number $\delta_{1}>0$, we can choose $\varepsilon_{1}\in \left[0,\varepsilon_{0}\right]$ such that  for every $\varepsilon_{1}\in \left[0,\varepsilon_{0}\right]$ and for each   $t\in(0, K\varepsilon^{-2H\beta}]$  
\begin{equation}
\sup_{K\varepsilon^{-2H\beta}\leq t \leq T}\E\left|Y^{\varepsilon}_{t}-\overline{Y}_{t}\right|^{2}\leq \delta_{1}.
\end{equation}
This completes the proof.
\end{proof}

With Theorem \ref{t1}, it is easy to show the convergence in probability between solution processes $Y^{\varepsilon}_{t}$ to the original \eqref{per} and $\overline{Y}_{t}$ to the averaged \eqref{noper}.
\begin{corollary}\label{c1}
Let the  assumptions $\bf(A1)$ and $\bf(A2)$ hold. For a given arbitary small number $\delta_{2}>0$, there exists $\varepsilon_{2}\in[0,\varepsilon_{0}]$ such that for all $\varepsilon\in(0,\varepsilon_{2}]$, we have
\begin{equation}
\lim_{\varepsilon\to 0}\mathbb{P}\left( \sup_{K\varepsilon^{1-\beta}\leq t \leq T}\left|Y^{\varepsilon}_{t}-\overline{Y}_{t}\right|> \delta_{2} \right)=0
\end{equation}
where $\beta$  defined by Theorem \ref{t1} such that $\beta < \frac{1}{2H}$.
\end{corollary}
\begin{proof}
By Theorem \ref{t1} and the Chebyshev inequality, for any given number $\delta_{2}>0$, we can obtain
\begin{equation*}
\mathbb{P}\left( \sup_{K\varepsilon^{1-\beta}\leq t \leq T}\left|Y^{\varepsilon}_{t}-\overline{Y}_{t}\right|> \delta_{2} \right)\leq
\frac{1}{\delta_{2}^2}\E\left( \sup_{K\varepsilon^{1-\beta}\leq t \leq T}\left|Y^{\varepsilon}_{t}-\overline{Y}_{t}\right|^{2} \right)\leq\frac{C_{4}\varepsilon^{1-2H\beta}}{\delta_2^2}.
\end{equation*}
Let $\varepsilon\to 0$ and the required result follows.
\end{proof}

\begin{remark}
Corollary \ref{c1} means the convergence in probability between the  original solution $\left(Y^{\varepsilon}_{t}, Z^{\varepsilon}_{1,t}, Z^{\varepsilon}_{2,t}\right)$  
 and the averaged solution $\left(\overline{Y}_{t}, \overline{Z}_{1,t}, \overline{Z}_{2,t}\right)$.
\end{remark}

\end{document}